\documentclass[11pt,a4paper]{amsart}
\usepackage[utf8]{inputenc}
\usepackage{faktor}
\usepackage[T1]{fontenc}
\usepackage{lmodern}
\usepackage[english]{babel}
\usepackage{amsmath}
\usepackage{xfrac}
\usepackage{esint}
\usepackage{stmaryrd,mathrsfs,bm,amsthm,mathtools,yfonts,amssymb,color,braket,booktabs,graphicx,graphics,amsfonts}
\usepackage{latexsym,microtype,indentfirst,hyperref}
\usepackage{xcolor}
\usepackage{courier}
\usepackage[colorinlistoftodos]{todonotes}
\usepackage{url}

\title{ Quantitative $W^{2, \, p}$-stability for almost Einstein hypersurfaces }
\author{ Stefano Gioffrè}

\theoremstyle{reference}
\newtheorem{teo}{Theorem}[section]
\newtheorem{prop}[teo]{Proposition}
\newtheorem{cor}[teo]{Corollary}
\newtheorem{lemma}[teo]{Lemma}
\numberwithin{equation}{section}

\newcommand{\enne}{\mathbb{N}}
\newcommand{\erre}{\mathbb{R}}
\newcommand{\esse}{\mathbb{S}}

\newcommand{\B}{\mathbb{B}}

\newcommand{\D}{\mathbb{D}}

\newcommand{\daA}[2]{\colon #1 \longrightarrow #2}

\newcommand{\restr}[2]{\left. #1 \right|_{#2}}
\newcommand{\dde}[1]{\frac{\partial}{\partial #1}}
\newcommand{\ddt}{\frac{d}{dt}}

\newcommand{\tRic}{\mathring{\Ric}}
\newcommand{\Rm}{\overline{R}}

\newcommand{\nablag}{_{g}\nabla}
\newcommand{\hessg}{_{g}\nabla^2}
\newcommand{\Ch}{\mathcal{C}}
\newcommand{\nomizu}{\owedge}
\newcommand{\gu}{\mathfrak{u}}
\newcommand{\gf}{\mathfrak{f}}
\newcommand{\gv}{\mathfrak{v}}
\newcommand{\gw}{\mathfrak{w}}
\newcommand{\Gammag}{_{g}\Gamma}

\renewcommand{\epsilon}{\varepsilon}
\renewcommand{\phi}{\varphi}
\renewcommand{\theta}{\vartheta}
\renewcommand{\ni}{\nu}

\DeclarePairedDelimiter{\abs}{|}{|}
\DeclarePairedDelimiter{\coup}{(}{)}				

\DeclarePairedDelimiter{\norm}{\lVert}{\rVert}

\DeclareMathOperator{\id}{id}

\DeclareMathOperator{\R}{R}
\DeclareMathOperator{\Ric}{Ric}
\DeclareMathOperator{\Riem}{Riem}
\DeclareMathOperator{\tr}{tr}

\DeclareMathOperator{\sgn}{sgn}
\DeclareMathOperator{\divv}{div}
\DeclareMathOperator{\Vol}{Vol}

\DeclareMathOperator{\grad}{grad}

\DeclareMathOperator{\diam}{diam}
\DeclareMathOperator{\Sec}{Sec}
\DeclareMathOperator{\injrad}{injrad}

\date{}
\allowdisplaybreaks
\begin{document}
\begin{abstract}
Let $n \ge 3$, $p \in (1, \, +\infty)$ be given. Let $\Sigma$ be a $n$-dimensional, closed hypersurface in $\erre^{n+1}$. It is a well-known fact that if $\Sigma$ is an Einstein hypersurface with positive scalar curvature, then it is a round sphere. We give the stable version of this result: if a hypersurface is almost Einstein in a $L^p$-sense, then it is $W^{2, \, p}$-near to a sphere. We give a quantitative estimate of this fact.
\end{abstract}
\maketitle
\section{Introduction}
Let $\Sigma$ be a $n$-dimensional, closed submanifold in $\erre^{n+1}$. We say that $\Sigma$ is an Einstein manifold if the traceless Ricci tensor 
\begin{equation}\label{Einstein}
\tRic = \Ric - \frac{1}{n} R g
\end{equation}
is identically $0$. In the '$30$s Thomas (see \cite{Thomas}) and Fialkov (see \cite{Fialkow}) independently proved that an Einstein hypersurface $\Sigma$ in $\erre^{n+1}$ with positive scalar curvature is isometric to the round sphere. However, the stability properties of this result are still unclear. Recently there have been attempts to prove such results. For example, in \cite{Roth} the author shows that if the Ricci tensor of a hypersurface $M \subset \erre^{n+1}$ is $L^\infty$-near to a constant, then there exists a diffeomorphism $F$ between $M$ and the round sphere $\esse^n$ whose differential $dF$ is $L^\infty$-near to the identity. Still, the result shown in \cite{Roth} is qualitative. In this paper we show that under certain geometric assumptions a hypersurface with small $L^p$-norm of the traceless Ricci tensor $\tRic$ is $W^{2, \, p}$-near to a round sphere. Let us state the theorem. Here we denote:
\begin{center}
\begin{tabular}{ll}
$\Vol_n$ & $n$-dimensional Hausdorff measure.\\
$\esse^n$ & standard sphere in $\erre^{n+1}$. \\
$\sigma$ & standard metric on the sphere. \\
$g$ & restriction of the $\erre^{n+1}$-flat metric to $\Sigma$.   \\
$A$ & second fundamental form for $\Sigma$.  \\
$\diam \Sigma$ & diameter of $\Sigma$ \\
$\id$ & identity function from a set to itself.   
\end{tabular}
\end{center}
We also say that a hypersurface is \textit{convex} if it is the boundary of an open, convex set.
\begin{teo}\label{MainThm}
Let $\Sigma$ be a closed, convex $n$-dimensional hypersurface in $\erre^{n+1}$ with volume equal to the round sphere $\esse^n$. Let also $p \in (1, \, \infty)$ be given.  There exists a positive number $\delta$ depending on $n$, $p$, $\norm{A}_\infty$, $\diam \Sigma$ with the following property. If $\norm{\tRic}_{L^p_g(\Sigma)} \le \delta$, then there exist a smooth diffeomorphism $\psi \daA{\esse^n}{\Sigma}$, a vector $c=c(\Sigma) \in \erre^{n+1}$ and a constant $ C= C(n, \, p, \,  \norm{A}_\infty, \, \diam \Sigma)  $ such that
\begin{equation}\label{MainThmEq}
\norm{\psi - \id - c}_{W^{2, \, p}_\sigma(\esse^n)} \le C \norm{\tRic}_{L^p_g(\Sigma)}
\end{equation}
\end{teo}
From theorem \ref{MainThm} we infer the following corollary.
\begin{cor}\label{MainCor}
Under the assumptions of theorem \ref{MainThm} there exist $\delta$ depending on $n$, $p$, $\norm{A}_\infty$, $\diam \Sigma$ with the following property. If $\norm{\tRic}_{L^p_g(\Sigma)} \le \delta$, then there exist a smooth diffeomorphism $\psi \daA{\esse^n}{\Sigma}$ and a constant $ { C= C(n, \, p, \,  \norm{A}_\infty, \, \diam \Sigma) }$ such that
\begin{equation}
\norm{\psi^*g - \sigma}_{W^{1, \, p}_\sigma(\esse^n)} \le C \norm{\tRic}_{L^p_g(\Sigma)}
\end{equation}
\end{cor}
Corollary \ref{MainCor} is particularly interesting. Indeed, in the theory of convergence for Riemannian manifolds (\textit{cf.} \cite{PetersenCollection}, \cite[Ch.$10$]{Petersen}, \cite{CheegerEbin}) there are many results about the $W^{k, \, p}$-nearness of a metric $g$ to a constant curvature one. However, these results have all qualitative nature. Corollary \ref{MainCor} provides a quantitative estimate instead, and to our knowledge it is the first result of this type. 
The proof of theorem \ref{MainThm} is quite long and contains different propositions which have interest in itself, so we have split the proof into four main proposition. The first proposition is a generalization of an almost Schur lemma proved by De Lellis and Topping in \cite{DLT}, and show how in a manifold we control the scalar curvature minus its average with its traceless Ricci tensor. In the second proposition we introduce a new technique, whose aim is to reduce particular tensorial problems to polynomial computations. We use it and find stable estimates for the Weyl tensor of a hypersurface. The actual proof of the main theorem actually starts in the third proposition. Here we apply a well-known theorem proved by Cheeger in \cite{Cheeger} and find a first, qualitative $C^{1, \, \alpha}$-stability result for $\Sigma$ to be near to a sphere. We make it quantitative in the fourth and last proposition, where we find a quantitative but non optimal $W^{2, \, p}$-estimate for $\Sigma$ to be near to a sphere. Using these results, we use a technique successfully applied in \cite{Gioffre2016} and optimise the fourth proposition's estimate.
\subsection*{Notations and preliminary theorems}
Throughout this paper we will use the previous notational conventions plus the following ones:
\begin{center}
\begin{tabular}{ll}
$n$ & Integer $ \ge 3$. \\
$p$ & real number in $(1, \, +\infty)$. \\ 
$\delta$  & standard metric in $\erre^n$. \\
$D$ & usual derivative in $\erre^n$. \\
$\Sigma$ & closed, convex $n$-dimensional hypersurface in $\erre^{n+1}$. \\
$\Riem^g$ & Riemann tensor associated to the metric $g$. \\ 
$\Sec^g$ & Sectional curvature associated to $\Riem^g$ . \\
$\Ric^g$ & Ricci tensor associated to the metric $g$. \\ 
$R^g$ & scalar curvature associated to the metric $g$. \\
$\nablag$ & Levi-Civita connection associated to a metric $g$.   \\
$\hessg$ & Hessian of a function (or a tensor) associated to the metric $g$. \\
$\Delta_g$ & Laplace-Beltrami operator associated to the metric $g$. \\
$\Gamma(E)$ & space of smooth sections of a vector bundle $E \rightarrow M$. \\
$O_{\epsilon^{\gamma}}(\norm{T}_{k, \, p})$ & scalar or tensorial quantity which satisfies the estimate \\
 &  $\norm{O_{\epsilon^{\gamma}}(\norm{T}_{k, \, p})}_{k, \, p} \le C \epsilon^{\gamma} \norm{T}_{k, \, p}$ for some constant $C$.
\end{tabular}
\end{center}
Whenever possible, we will omit the subscripts. Below we give a list of known facts and theorems we are going to use in the paper.
\subsubsection*{Geometric quantities}
We fix the sign convention for the main geometric quantities we are going to study in the paper. 
We define 
\[
\R(X, \, Y)Z := \nabla^2_{Y, \, X} Z - \nabla^2_{X, \, Y} Z
\]
The Riemann curvature is the $4$-covariant tensor given by lowering one index in the previous expression.
\begin{equation}\label{RiemannSign}
\Riem(X, \, Y, \, Z, \, W) = \coup*{\R(X, \, Y)Z, \, W}
\end{equation}
The Ricci curvature is the $2$-covariant tensor given by taking the $(1, \, 3)$-trace of the Riemann curvature. Namely,
\begin{equation}\label{RicciSign}
\Ric_{ij} := g^{pq} \Riem_{ipjq}
\end{equation}
Finally, the scalar curvature is given by taking the trace of the Ricci curvature.
\begin{equation}\label{CurvatureSign}
R= g^{ij} \Ric_{ij}
\end{equation}
We recall the following well known corollary of the differential Bianchi identity (see \cite[p. $184$]{GHL}), which relates the derivatives of the Ricci curvature with the derivatives of the scalar curvature.
\begin{lemma}\label{DiffBian}
Let $(M, \, g)$ be a Riemannian manifold. The following equation holds:
\begin{equation}\label{DiffBianEq}
\divv \Ric = \frac{1}{2} \nabla R
\end{equation}
\end{lemma}
Another important quantity we need to recall is the second fundamental form. Let $\Sigma$ be an oriented, hypersurface in $\erre^{n+1}$, and let $\nu$ be its outer normal. We define $A$ as
\begin{equation}\label{2ndFF}
A(X, \, Y) := \coup*{X, \, D_Y \nu} \mbox{ for any } X, \, Y \in \Gamma(\Sigma)
\end{equation}
It is straightforward to see that $A$ is a well-defined, symmetric, $2$-covariant tensor. The second fundamental has a crucial role in Differential Geometry. We recall here two theorems we will use in the paper. Firstly, we recall the Gauss equations for hypersurfaces in an Eucliden space. (see \cite[p. $248$]{GHL}
\begin{teo}\label{GaussThm}
Let $\Sigma$ be a hypersurface in $\erre^{n+1}$, and let $g$ be its induced metric. Then, the following equation holds:
\begin{equation}\label{GaussEq}
\Riem_{ijkl} = \frac{1}{2} A \nomizu A_{ijkl} = A_{ik} \, A_{jl} - A_{il} \, A_{jk}
\end{equation}
\end{teo} 
Secondly, we recall the following $n$-dimensional version of the famous \textit{Nabelpunksatz}. (See \cite[p. $8$]{Spivak})
\begin{teo}\label{Nabel}
Let $\Sigma$ be a closed hypersurface in $\erre^{n+1}$, with induced metric $g$. Assume that $A=g$. Then $\Sigma$ is isometric to the round sphere $\esse^n$.
\end{teo}
We need to recall the following proposition, (see \cite[p. $48$]{Daniel}) which shows how convexity affect the Ricci tensor and the second fundamental form.
\begin{prop}\label{Convexity}
Let $n \ge 2$, and be given, and let $\Sigma$ be a closed, connected $n$-dimensional hypersurface in $\erre^{n+1}$. Then the following facts are equivalent: 

\begin{tabular}{ll}
$(i)$ &  $\Ric \ge 0$  everywhere on $\Sigma$. \\
$(ii)$ & $A \ge 0$  everywhere on $\Sigma$.
\end{tabular}

In particular, if either of the one conditions above holds, $\Sigma$ is convex.
\end{prop}
\subsubsection*{Kulkarni-Nomizu product}  Given two $2$-covariant symmetric tensors $A$ and $B$, we denote by $A \nomizu B$ the $4$-covariant symmetric tensor defined by
\begin{equation}\label{Nomizu}
A \nomizu B_{ijkl} := A_{ik} \, B_{jl} + A_{jl} \, B_{ik} - A_{il} \, B_{jk} - A_{jk} \, B_{il} \quad 
\end{equation}
$A \nomizu B$ is called the \textit{Kulkarni-Nomizu product of $A$ and $B$}. The importance of such quantity is due to the following Riemann decomposition theorem (see \cite[p. $182$-$183$]{GHL}).
\begin{teo}\label{RiemDecomp}
Let $n \ge 4$ be given, and let $(M, \, g)$ be a $n$-dimensional manifold. Then its Riemann tensor can be decomposed as follows:
\begin{equation}\label{RiemDecompEq}
\Riem = \frac{R}{2n(n-1)} g \nomizu g + \frac{1}{n-2} \tRic \nomizu g + W
\end{equation}
where $W$ denotes the $4$-covariant Weyl tensor. This decomposition is orthogonal, namely 
\begin{equation}\label{RiemDecompNorm}
\abs{\Riem}^2 = \frac{R^2}{4n^2(n-1)^2} \abs{g \nomizu g}^2 + \frac{1}{(n-2)^2} \abs{\tRic \nomizu g}^2 + \abs{W}^2
\end{equation}
If $n=3$, then the Weyl tensor vanishes.
\end{teo}
We immediately notice that due to the Schur lemma, if both $W$ and the $\tRic$ vanish, then $(M, \, g)$ has constant sectional curvature.
\subsubsection*{Cheeger-Gromov convergence}  Given three positive numbers $\Lambda$, $V$ and $D$, we define the Cheeger-Gromov class.
\begin{equation}\label{CheegerGromov}
\Ch_{\Lambda, \, V, \, D} :=   \Set{ (M, \, g)  \mid \norm{\Riem}_\infty \le \Lambda,\ V \le \Vol(M), \, \diam M \le D}
\end{equation}
The Cheeger-Gromov class plays an important role in the theory of convergence for Riemannian manifolds due to its compactness property. (see \cite{CheegerEbin} or \cite{Anderson})
\begin{teo}\label{CheegerCptness}
Let $\set{(M_k, \, g_k)}_{k \in \enne} \subset \Ch_{\Lambda, \, V, \, D}$ be given. For every $ {\alpha \in (0, \, 1)}$ there exist $(M, \, g) \in \Ch_{\Lambda, \, V, \, D}$ and a subsequence $\set{(M_{k_h}, \, g_{k_h})}_{k_h \in \enne}$ that converges in $C^{1, \, \alpha}$ to $(M, \, g)$. For convergence we mean that for every $k_h$ there exists a diffeomorphism 
\[
\phi_{k_h} \daA{M_{k_h}}{M}
\]
such that the pull-back metrics $\phi_{k_h}^* g_{k_h}$ converge to $g$ in $C^{1, \, \alpha}$.
\end{teo}
\subsubsection*{Injectivity radius and harmonic coordinates} The importance of the Cheeger-Gromov class is not only restricted to its compactness property. Riemannian manifolds in the class have a lower bound on their injectivity radius $\injrad_g(M)$. Indeed, in \cite{CheegerEbin} Cheeger proved the following lemma.
\begin{lemma}\label{InjControl}
There exists $i_0=i_0(n, \, \Lambda, \, V, \, D)>0$ such that
\begin{equation}
\injrad_g(M) \ge i_0 \mbox{ for every } (M, \, g) \in \Ch_{\Lambda, \, V, \, D}
\end{equation}
\end{lemma}
We will need to use this lemma in combination with the so-called \textit{harmonic coordinates}.  We recall the definition: given a manifold $(M, \, g)$ and an open set $U \subset M $ a mapping $y \daA{U}{\erre^{n+1}}$ is said to be \textit{harmonic} if it is a diffeomorphism and if it satisfies the equation 
\[
\Delta_g y = 0
\]
A detailed study on harmonic coordinates can be found in \cite[p.$304$-$307$]{Petersen} or in \cite[p.$523$]{Jost}. We just recall the following theorem\footnote{Actually, the result shown in \cite[p.$523$]{Jost} is more general and slightly different, but easily implies this formulation.}
\begin{teo}\label{Harmonic}
Let $(M, \, g)$ be a $n$-dimensional manifold, and let $x \in M$, $\alpha \in (0, \, 1)$ be given. Assume the sectional curvature of $M$ to be bounded by a positive number $\Lambda$. There exist $R_0=R_0(n, \, \Lambda, \, \injrad(x))$ and $C=C(\alpha, \, n, \, \Lambda R_0)$ and harmonic coordinates $y \daA{B_{R_0}(x)}{\erre^n}$ such that in the frame $\Set{\dde{y^1} \, \dots \, \dde{y^n}}$ the metric $g$ satisfies
\begin{equation}\label{HarmonicEq}
\norm{g - \delta}_{C^{1, \, \alpha}(B_{R_0(x)})} \le C 
\end{equation}
\end{teo}
Harmonic coordinates are also useful because they give a nice expression for the Ricci operator. Indeed, the following expression holds:
 \begin{equation}\label{RicciHarm}
 -\frac{1}{2} \Delta_{g} g_{ij} + Q_{ij}(g, \, D g) = \Ric^{g}_{ij} \mbox{ for every indexes } i, \, j
 \end{equation}
 where $Q_{ij}$ is a universal polynomial depending on $g$ and its first derivatives $Dg$. The computations can be found in \cite[p. $305$-$307$]{Petersen}.
\subsubsection*{Obata's theorem}
We state here the last ingredient of the section, the Obata's theorem, which gives us the exact value of the first eigenvalue of the spherical laplacian and a characterization of its kernel.
\begin{teo}\label{Obata}
Let $(M, \, g)$ be a closed manifold which satisfies the following condition on the Ricci tensor $\Ric$:
\begin{equation}
\Ric(X, \, X) \ge (n-1) \, g(X, \, X) \mbox{ for every vector field } X
\end{equation}
and the Laplacian condition:
\begin{equation}
-\Delta_g f = n  f \mbox{ for some } f
\end{equation}
Then $M$ is isometric to the round sphere $(\esse^n, \, \sigma)$ and we also have
\begin{equation}
\ker \Delta_\sigma + n = \set{ \phi_v  \mid \phi_v(x) = (v, \, x),\ v\in \erre^{n+1}}
\end{equation}
\end{teo}
We also need a result proved in \cite{Gioffre2016} which shows how for a function  in the sphere the $L^p$-norm of $\Delta_\sigma f + n f$ ''almost controls'' the $W^{2, \, p}$-norm of $f$.
\begin{lemma}\label{ObataLike}
Let $f \in C^\infty(\esse^n)$ be given. Define $v_f \in \erre^{n+1}$ as
\[
v_f := (n+1)\fint_{\esse^n} z f(z) \, dV_\sigma
\]
Then, the following estimate holds:
\begin{equation}\label{ObataLikeEq}
\norm{f - \phi_{v_f}}_{W^{2, \, p}_\sigma(\esse^n)} \le C(n, \,p) \norm{\Delta_\sigma f + n f}_{L^p_\sigma(\esse^n)}
\end{equation} 
\end{lemma}
\subsubsection*{Radially parametrized hypersurfaces}
We exhibit the parametrization on which we will work. Let us assume for a moment that $\Sigma$ is the border of an open, convex set $U$ containing $0$. We can give the following radial parametrization for $\Sigma$:
\begin{equation}\label{PsiDef}
\psi \daA{\esse^n}{\Sigma },\ \psi(x) := \rho(x) \, x := e^{f(x)} \, x
\end{equation}
Clearly $\psi$ is a smooth diffeomorphism. If $U$ does not contain $0$ we can still give such parametrization by properly translating $U$. We will say that $\Sigma$ is \textit{radially parametrized} if it can be written as the image of such $\psi$. We call $\rho$ the \textit{radius} of $\Sigma$ and $f$ the \textit{logarithmic radius} of $\Sigma$.
We will also say that a hypersurface $\Sigma$ is \textit{admissible} if satisfies the following:
\begin{align}
\norm{A}_\infty &\le \Lambda \label{2ffBound} \\
\Vol_n(\Sigma) &= \Vol_n(\esse^n) \label{VolEq} \\
\diam \Sigma &\le D \label{DiamBound}
\end{align}
We notice that the Gauss equation \eqref{GaussEq} tells us that the bound on the second fundamental form implies a bound on the Riemann tensor. Indeed, with a simple computation we obtain
\[
\abs{\Riem}^2 = \frac{1}{4} \abs{A \nomizu A}^2 =  \abs{A}^4 - \abs{A^2}^2 \le \Lambda^4
\]
Hence an admissible hypersurface lies in the Cheeger-Gromov class $\Ch_{\Lambda^2, \, V, \,D}$.
We recall the expression for the main geometrical quantities of a radial parametrized hypersurface. It is important to remark that here and later in paper, we use the letter $g$ to denote both the pull-back metric $\psi^*g$ and the metric $g$ on $\Sigma$. Although this is an abuse of notation, it strongly simplifies the notation.
\begin{lemma}\label{Computations}
Let $\psi$ be as in \eqref{PsiDef}. Then we have the following expressions:
\begin{align}
g_{ij} &= e^{2 f} \coup*{\sigma_{ij} + \nabla_i f \, \nabla_j f } \label{g} \\
g^{ij} &= e^{- 2 f} \coup*{\sigma_{ij} - \frac{\nabla_i f \, \nabla_j f}{1 + \abs{\nabla f}^2} } \label{ginv} \\
\ni(x) &= \frac{1}{\sqrt{1 + \abs{\nabla f}^2}} \coup*{x - \grad_\sigma f(x)} \label{norm} \\
A_{ij} &= \frac{e^f}{\sqrt{1 + \abs{\nabla f}^2}} \coup*{ \sigma_{ij} + \nabla_i f \, \nabla_j f - \nabla^2_{ij} f} \label{Aff} \\
A^i_j 
&= \frac{e^{-f}}{\sqrt{1 + \abs{\nabla f}^2}} \coup*{ \delta^i_j - \nabla^i \nabla_j f + \frac{1}{1 + \abs{\nabla f}^2} \nabla^i f \, \nabla^2 f [\nabla f]_j} \label{Aalta} \\
H &:= g^{ij} A_{ij} = e^{-f} \coup*{ \frac{n}{\sqrt{1 + \abs{\nabla f}^2}} - \divv_\sigma \coup*{\frac{\nabla f}{\sqrt{1 + \abs{\nabla f}^2}}}} \label{MeanC} \\
dV_g &= e^{nf} \, \sqrt{1 + \abs{\nabla f}^2} \, dV_\sigma \label{Volume} \\
\Gammag_{ij}^k 
&= \Gamma_{ij}^k + \frac{1}{1 + \abs{\nabla f}^2} \nabla^2_{ij} f \, \nabla^k f + \coup*{\nabla_i f \, \delta^k_i + \nabla_j f \, \delta^k_i - \nablag^k f \, g_{ij}}   \label{Christoffel} 
\end{align}
\end{lemma}
The computation of lemma \ref{Computations} is given in \cite[p. $15$-$17$]{Gioffre2016}.
\section{Proof of the main theorem}
We state here the propositions which lead to the theorem's proof, and then prove \ref{MainThm}.
\begin{prop}\label{AverageScalarThm}
Let $(M, \, g)$ be a $n$-dimensional, closed manifold be in the Cheeger-Gromov class $ \Ch_{\Lambda, \, V, \, D}$ with $ \Vol_g(M) = V$.  There exists a constant $C=C(n, \, p, \, \Lambda, \, V, \, D)>0$ such that the following estimate holds:
\begin{equation}\label{AverageScalarEq}
\norm{R - \overline{R}}_{L^p_g(M)} \le C \norm{\tRic}_{L^p_g(M)}
\end{equation}
\end{prop}
\begin{prop}\label{WeylThm}
Let $\Sigma$ be an admissible hypersurface in $\erre^{n+1}$. There exists a constant $C=C(n, \, p, \, \Lambda, \, V, \, D)>0$ such that the following estimate holds:
\begin{equation}\label{WeylEq}
\norm{W_g}_{L^p_g(\Sigma)} \le C \coup*{ \norm{\tRic}_{L^p_g(\Sigma)} + \abs{ \Rm - n(n-1)} } 
\end{equation}
where $n(n-1)$ is the scalar curvature $R_\sigma$ of the unit round sphere.
\end{prop}
\begin{prop}\label{ApproxThm}
Let $\Sigma$ be a radially parametrized, admissible hypersurface in $\erre^{n+1}$ and let $\alpha \in (0, \, 1)$ be given. For every $\epsilon>0$ there exists a positive number $\delta=\delta(n, \, p, \, \Lambda, \, D, \, \alpha)$ such that if the traceless Ricci tensor and the Weyl tensor are small, namely 
\[
\norm{\tRic}_{L^p_g(\Sigma)} + \norm{W}_{L^p_g(\Sigma)} \le \delta
\]
then the pull-back metric is $C^{1, \, \alpha}$-near to the round sphere metric, namely 
\begin{equation}\label{ApproxEq}
\norm{\psi^*g - \sigma}_{C^{1, \, \alpha}(\esse^n)} \le \epsilon
\end{equation}
In particular, the logarithmic radius $f$ satisfies the estimate:
\begin{align}
\norm{f}_\infty &\le C \epsilon  \label{ApproxFirst}\\
\norm{\nabla f}_\infty &\le C \sqrt{\epsilon}  \label{ApproxSecond} \\
 \sup \abs{\nabla f} \, \abs{\nabla^2 f}  &\le C \sqrt{\epsilon} \label{ApproxThird} 
\end{align}
for some constant $C$ not depending on $\epsilon$.
\end{prop}
\begin{prop}\label{AlmostThm}
Let $\Sigma$ be a radially parametrized, admissible hypersurface in $\erre^{n+1}$. Assume that the logarithmic radius $f$ satisfies estimate \eqref{ApproxFirst}, \eqref{ApproxSecond}, \eqref{ApproxThird}. Then, there exist a constant $C=C(n, \, p, \,  \Lambda, \,  D)>0$ such that the following estimate holds:
\begin{equation}\label{AlmostEq}
\norm{f - (v_f, \, \cdot)}_{W^{2, \, p}_\sigma(\esse^n)} \le C \coup*{\norm{\tRic}_{L^p_g(\Sigma)} + \sqrt{\epsilon}\norm{f}_{W^{2, \, p}_\sigma(\esse^n)}  }
\end{equation}
\end{prop}
We show how theorem \ref{MainThm} follows by these propositions. The proof we are giving is essentially copied from \cite[p. $4$-$5$]{Gioffre2016}, and shows how to optimise estimate \eqref{AlmostEq} by removing $v_f$.
\begin{proof}[Proof of theorem \ref{MainThm}]
Let $\epsilon>0$ be fixed for the moment. At the end of the argument we will choose it small enough. Let $\Sigma$ be  an admissible, radially parametrized hypersurface, and let $\delta$ be given so that the logarithmic radius $f$ satisfies inequalities \ref{ApproxFirst}, \ref{ApproxSecond}, \ref{ApproxThird} and \ref{AlmostEq}.
We notice that for every $c \in U$ we can define 
\begin{equation}
\psi_c \daA{\esse^n}{\Sigma - c },\ \psi_c(x) := \rho_c(x) \, x := e^{f_c(x)} \, x
\end{equation}
For every $c$ the mapping $\psi_c$ is an alternative radial parametrization for $\Sigma$, and it is a well defined diffeomorphism. We can also define:
\begin{equation}
\Phi \daA{U}{\erre^{n+1}},\ \Phi(c):= - (n+1) \fint_{\esse^n}   z f_c(z) \, dV_\sigma(z)
\end{equation}
Our idea is to find $c_0 \in U$ such that $\Phi(c_0)=0$. Then we are done, because for such $f_{c_0}$ we obtain the estimate 
\[
 \norm{f_{c_0}}_{W^{2, \, p}_\sigma(\esse^n)} \le C \coup*{ \norm{\tRic}_{L^p_g(\Sigma)} + \sqrt{\epsilon} \, \norm{f_{c_0}}_{W^{2, \, p}_\sigma(\esse^n)} }  
\]
Therefore if we set $\epsilon_0 = \frac{2}{C}$ we can find a $\delta_0$ given by proposition \ref{ApproxThm} such that the last term can be absorbed, namely
\[
\norm{\tRic}_{L^p_g(\Sigma)} + \norm{W}_{L^p_g(\Sigma)} \le \delta \Rightarrow \norm{f_{c_0}}_{W^{2, \, p}_\sigma(\esse^n)} \le C \norm{\tRic}_{L^p_g(\Sigma)} 
\]
This estimate proves theorem \ref{MainThm} with $c=c_0$.

Let us find such $c_0$. By the hypothesis on the logarithmic radius, we can easily find a radius $r=r(\Sigma)>0$ such that for every $c \in \D_r$ we have that $f_c$ still satisfies such estimates. We work with $H$ inside the disk $\D_r$. 
We start with the following simple consideration: for every $z \in \esse^n$ there exists $x_c=x_c(z)$ in $\esse^n$ so that
\[
\psi_c(z) = \psi(x_c) - c
\]
We expand this equality and find
\[
\rho_c(z) \, z = \rho(x_c) \, x_c - c 
\]
We take the absolute value and obtain that $\rho_c$ satisfies the equality:
\[
\rho_c(z) = \abs{\rho(x_c) \, x_c - c}
\]
while $x_c=x_c(z)$ satisfies the relation
\begin{equation}\label{xc}
x_c(z) = \frac{\rho_c(z) \, z + c}{\rho(x_c(z))}
\end{equation}
Using the $W^{1, \, \infty}$-smallness of $f$, we approximate $\rho$ and $\rho_c$, and find
\begin{equation}\label{xcApp}
x_c(z) = \coup{ z + c} \coup{1 + o_\epsilon(1)}
\end{equation}
We approximate $f_c$:
\begin{align*}
f_c(z) &= \frac{1}{2} \, \log \abs*{ e^{f(x_c) } x_c - c }^2 = \frac{1}{2} \log \coup*{ e^{2 f(x_c)} - 2 e^{f(x_c)} (x_c, \, c) + \abs{c}^2 }^2 \\
&= \frac{1}{2} \log \coup*{1 + \abs{c}^2 - 2(x_c(z), \, c) + o_\epsilon(1)} \\
&= \abs{c}^2 - 2(x_c(z), \, c) + o(\abs{c}) + o_\epsilon(1) =  - 2 (z, \, c) +  o(\abs{c}) + o_\epsilon(1)
\end{align*}
This allows us to write $\Phi$ as follows:
\begin{align*}
\Phi_i(c) 
&=  (n+1 )\fint_{\esse^n}  (x_c(z), \, c)  \, z_i \, dV_\sigma + o(\abs{c}) + o_\epsilon(1)   \\
&=  (n+1 )\fint_{\esse^n}  ((z+c)(1 + o_\epsilon(1)), \, c)  \, z_i \, dV_\sigma + o(\abs{c}) + o_\epsilon(1)    \\
&=   (n+1) \coup*{\int_\esse^n \abs{z_i}^2 \, dV_\sigma}c  + o(\abs{c}) + o_\epsilon(1) \\
&=    c  + o(\abs{c}) + o_\epsilon(1) 
\end{align*}
We are now able to show that $0$ is in the range of $\Phi$. We restrict $\Phi$ to $\D_r$ and finally we choose $\epsilon$ and $r$ so small that 
\[
\abs{\Phi(x) - x} < \frac{1}{2} \mbox{ for every } x \in \D_r
\]
Let us argue by contradiction: suppose that $0 \notin R(\Phi)$, then we can consider $\overline{\Phi}:= \frac{\Phi}{\abs{\cdot}} \daA{\esse^n}{\esse^n}$, and notice that 
\begin{equation}\label{Homotopic}
\abs{\overline{\Phi}(x) - x} < 2 \mbox{ for every } x \in \esse^n
\end{equation}
It is easy to see that if $c_n \to c_0$ then $f_{c_n} \to f_{c_0}$ pointwise and the family $\set{f_c}_{c \in \D_r}$ is equibounded. This proves that $\Phi$ and therefore also $\Phi$ are continuous. However, estimate \eqref{Homotopic} tells us that $\overline{\Phi}$ is homotopic to the identity; but at the same time, we obtain that $\overline{\Phi}$ is the restriction of a continuous map defined on the ball, hence it cannot be homotopic to the identity.
\end{proof}
The rest of the article will be devoted to proving the three propositions.
\section{Proof of the propositions}
\subsection{Proof of proposition \ref{AverageScalarThm}}
We prove theorem \ref{AverageScalarEq}. The idea of the proof is to study the differential Bianchi equation as a partial differential equation between the scalar curvature and the traceless Ricci tensor. As for the proof of theorem \ref{MainThm}, we essentially copy a technique used in \cite{Daniel} and then in \cite{Gioffre2016}, where the authors deal with a similar differential equation on hypersurfaces and on the sphere respectively.

Our starting point is the Bianchi identity \eqref{DiffBianEq}. Writing $\Ric = \tRic + \frac{1}{n} R g$ we obtain the differential relation 
\begin{equation}\label{DiffBianTraceless}
\nabla R = \frac{2n}{n-2} \divv \tRic
\end{equation}
We let the thesis follow from the following lemma.
\begin{lemma}\label{AnArmLemma}
Let $(M, \, g)$ be a closed, $n$-dimensional manifold in the Cheeger-Gromov class $\Ch_{\Lambda, \, V, \, D}$. Let
 $\gu \in C^\infty(M)$, $\gf \in \Gamma(T^*M \otimes T^*M)$, be given so that the following equation holds:
\[
\nabla \gu = \divv_g \gf 
\]
There exists $C=C(n, \, p, \, \Lambda, \, \Vol(M), \, D)$ such that the following estimate holds:
\begin{equation}\label{AnArm}
\norm{\gu - \overline{\gu}}_{L^p_g(M)} \le  C \norm{\gf}_{L^p_g(M)}
\end{equation}
where we have set 
\[
\overline{\gu}:= \fint_{\esse^n} \gu \, dV_g
\]
\end{lemma}
\begin{proof}
Let us assume $\Vol_g(M)=V$ without loss of generality. We fix $\alpha \in (0, \, 1)$ and patch together Cheeger's lemma \ref{InjControl}, theorem \ref{Harmonic} on harmonic coordinates, and easily obtain a positive number $R_0$ with the following property. For every $x \in M$, harmonic coordinates $y \daA{B_{R_0}(x)}{\erre^n}$ are well defined, and the metric $g$ in these coordinates satisfies the estimate
\[
\norm{g}_{C^{1, \, \alpha}(B_{R_0}(x))} \le C(n, \, \alpha, \, \Lambda, \, V, \, D ) R_0^2
\]
We expand the term in divergence form, and obtain 
\begin{align*}
\divv_g \gf_k 
&= g^{ij} \coup*{ D_i \gf_{jk} + \Gamma^l_{ij} \, \gf_{lk} + \Gamma_{ik}^l \, \gf_{lj}  } 
=  g^{ij} \coup*{ D_i \gf_{jk}  + \Gamma_{ik}^l \, \gf_{lj}  } \\
&= D_i \coup*{g^{ij} \gf_{ik}} - D_i g^{ij} \gf_{jk}  + \Gamma_{ik}^l \, \gf_{lj} =: \divv_\delta \tilde{\gf}_k + O(\gf)_k
\end{align*}
where in the last equality we have set $\tilde{\gf}^i_k := g^{ij} \gf_{jk} $, and used $O(\gf)$ to denote a quantity that satisfies the estimate 
\[
\abs{O(\gf)} \le C \abs{\gf} \mbox{ for some } C
\]
We argue as in \cite{Gioffre2016} and write $\gu = \gv + \gw$, where $\gv$ and $\gw$ satisfy the conditions:
\[
\begin{cases}
\Delta_\delta \gv = \divv_\delta \divv_\delta \tilde{\gf} \\
\restr{\gv}{\partial B_{R_0}(x)} = \restr{\gu}{\partial B_{R_0}(x)}
\end{cases}
\]
and 
\[
\begin{cases}
\Delta_\delta \gw = \divv \coup*{ \, \gf[h] + O_\epsilon(\gf) \, } \\
\restr{\gw}{\B_{R_0}(x)} = 0
\end{cases}
\]
where $\Delta_\delta$ is the flat laplacian. The first system is studied in \cite[p. $12$-$16$]{Daniel} where the author proves the existence of a real number $\lambda $ such that the following estimate holds: 
\[
\norm{\gv - \lambda(x)}_{L^p_\delta(B_{\faktor{R_0}{4}}(x))} \le C \norm{\gf}_{L^p_\delta(B_{R_0}(x))}
\]
where $C=C(n, \, p, \, R_0)$. The second system is well known. In \cite[p. $80$-$81$]{Ambrosio} it is shown the inequality:
\[
\norm{\gw}_{L^p_\delta(B_R(x))} \le C \norm{\gf}_{L^p_\delta(B_R(x))} 
\]
where $C=C(n, \, p, \, R_0)$. Both the integral estimates are made with the (local) flat measure, but using the local nearness of $g$ to $\delta$ we can easily notice that the measure $dx$ and $dV_g$ are equivalent, and the equivalence constants depend only on $n$, $p$, $\Lambda$, $R_0$.
We patch together the two estimates and obtain that there exists $r_0 = \faktor{R_0}{4}$ such that for every $x$ the following estimate holds:
\begin{equation}
\norm{\gu - \lambda(x)}_{L^p_g(B_{r_0}(x))} \le  C \norm{\gf}_{L^p_g(M)}
\end{equation}
where $\lambda(x)$ is a number depending only on $x$. We make the estimate global, by using the following lemma, proved in \cite[p. $8$-$9$]{Gioffre2016}\footnote{The proof in \cite{Gioffre2016} is given just in the case $(M, \, g) = (\esse^n, \, \sigma)$ but its extension is trivial.}  
\begin{lemma}
Suppose $\gu \in C^\infty(M)$ has the following property. There is a radius $\rho$ such that for every $x \in M$ the local estimate is satisfied:
\begin{equation}\label{LocalDue}
\norm{\gu - \lambda(x) }_{L^p_g(B_{r}(x))} \le  \beta
\end{equation}
where $\lambda(x)$ is a real number depending on $x$, $r \le 2\rho$ and $\beta$ does not depend on $x$. Then $\gu$ satisfies the global estimate:
\[
\norm{\gu - \lambda }_{L^p_g(M)} \le C \beta
\]
where $\lambda \in \erre$ and $C = C(n, \, p, \, \rho)$ is a positive constant. 
\end{lemma}
We complete the proof by showing that
\begin{equation}
\norm{\gu - \overline{\gu}}_p \le 2 \min_\lambda \norm{\gu - \lambda}_p
\end{equation}
Indeed, let us assume that the ambient space has unit volume, and let $\lambda_0 \in \erre$ be given so that
\[
\norm{\gu - \lambda_0}_p \le 2 \min_\lambda \norm{\gu - \lambda}_p
\]
We easily conclude the thesis.
\begin{align*}
\norm{\gu - \overline{\gu}}_p 
&\le  \norm{\gu - \lambda_0}_p + \abs{\overline{\gu} - \lambda_0}^{\frac{1}{p}} \le \norm{\gu - \lambda_0}_p + \abs*{\fint \gu  - \lambda_0}^{\frac{1}{p}} \\
&\le \norm{\gu - \lambda_0}_p + \coup*{\fint \abs{\gu  - \lambda_0}}^{\frac{1}{p}} \le 2 \norm{\gu - \lambda_0}_p  \\
&\le 2 \min_\lambda \norm{\gu - \lambda}_p
 \end{align*}
\end{proof}
\subsection{Proof of proposition \ref{WeylThm}}
We deal with proposition \ref{WeylThm}. In this proof we introduce a new technique, which we use to reduce geometric problems to polynomial ones.  We briefly explain the idea, which is very simple: when we have a tensorial identity between symmetric tensors, we study it pointwise and use the spectral theorem to diagonalise the main quantities. Diagonalising gives us an equality for the eigenvalues of these tensor, and therefore the identity is reduced to the study of the zeros of a polynomial. We exhibit the main results in this section.
\begin{prop}[Ridigity]\label{Rigidity}
Let $\Sigma$ be a closed, convex hypersurface in $\erre^{n+1}$. Assume that at every point we have the equality
\begin{equation}\label{RigidityEq}
\Ric = (n-1)g
\end{equation}
Then $\Sigma$ is isometric to the round sphere.
\end{prop}
\begin{prop}[Stability]\label{Stability}
Let $\Sigma$ be a closed, convex hypersurface in $\erre{n+1}$. There exists a constant $C=C(n, \, p, \, \Lambda, \, D)$ such that 
\[
\norm{W}_{L^p_g(\Sigma)} \le C \coup*{ \norm{\tRic}_{L^p_g(\Sigma)}  + \abs{\Rm-  n(n-1) } }
\] 
\end{prop}
We prove the propositions.
\begin{proof}[Proof of \ref{Rigidity}]
We consider equality \eqref{RigidityEq}. Tracking the Gauss equation \eqref{GaussEq} we obtain 
\begin{equation}\label{GaussRicci}
\Ric_{ij} = H A_{ij} - A_i^k A_{kj} 
\end{equation}
We consider a point $p$ and study the equality at $p$. Using the spectral theorem, we are able to find local coordinates such that $g= \delta$ and $A=D(x)$ is diagonal with eigenvalues $\set{x_1, \, \dots x_n}$ at $p$. We can rewrite equality \eqref{RigidityEq} as 
\[
\tr D(x) \, D(x) - D(x)^2 = (n-1) \delta
\]
In terms of the eigenvalues of $A$ we infer the following system:
\begin{equation}\label{PolyRig}
\coup*{\sum_{i=1}^n x_i  } x_j - x_j^2 = n-1 \mbox{ for every } j
\end{equation}
The thesis follows from the following lemma.
\begin{lemma}\label{LemmaPolyRig}
The only solutions of system \eqref{PolyRig} are the vectors $(1, \, \dots \, 1)$ and $(- 1, \, \dots \, -1)$.
\end{lemma}
The proof of lemma \ref{LemmaPolyRig} is in the last section. From this result and from the convexity of $\Sigma$ we obtain that necessarily $A=g$ everywhere on $\Sigma$. The thesis follows by theorem \ref{Nabel}.
\end{proof}
\begin{proof}[Proof of \ref{Stability}]
We recall the definition of Weyl tensor:
\[
W = \Riem - \frac{R}{2n(n-1)}  g \nomizu g + \frac{1}{n-2} \tRic \nomizu g
\]
Let us denote by $O_p(\tRic)$ a quantity that can be approximated by $\tRic$ in the $L^p$-norm. With this notation and proposition \ref{AverageScalarEq}, we write
\begin{align*}
W 
&= \Riem - \frac{\Rm}{2n(n-1)} g \nomizu g + O_p(\tRic) \\
&= \Riem - \frac{1}{2} g \nomizu g + O_p(\tRic) + O(\abs{\Rm - n(n-1)})
\end{align*}

Now we recall equation \ref{GaussEq}, and finally obtain the expression
\[
W = \frac{1}{2} (A-g) \nomizu (A+g) + O_p(\tRic) + \abs{\Rm - R_\sigma}
\]
which easily gives us the $L^p$-inequality
\begin{equation}\label{WeylIneq}
\norm{W}_{L^p_g(\Sigma)} \le C \coup*{ \norm{(A-g) \nomizu (A+g)}_{L^p_g(\Sigma)} + \norm{\tRic}_{L^p_g(\Sigma)} + \abs{\Rm - R_\sigma}}
\end{equation}
where $C$ basically depends on proposition \ref{AverageScalarEq}. Now we apply the same technique seen before, and study the quantity $(A - g) \nomizu (A + g)$ pointwise. Again, we fix $p \in \Sigma$ and consider coordinates such that $g= \delta$, $A=D(x)$ at $p$. We let the thesis follow by the following lemma. 
\begin{lemma}\label{LemmaPolyDue}
Define the polynomials
\begin{align}
p(x) &:= \abs{(D(x)-\delta) \nomizu (D(x)+\delta)}^2 \label{PolyNomizu} \\
q(x) &:= \abs{\Ric(x) - (n-1)\delta}^2 \label{PolyRicci}
\end{align}
where $\Ric(x)$ is the diagonal matrix given by \eqref{GaussRicci}.
There exist positive constants $c_0=c_0(n)$ and $c_1=c_1(n, \, \Lambda)$ such that 
\begin{equation}\label{Confront}
c_0 \le \frac{q(x)}{p(x)} \le c_1 \mbox{ in the ball } B_\Lambda
\end{equation}
\end{lemma}
Using lemma \ref{LemmaPolyDue} we find the thesis. Indeed, by \eqref{Confront} we obtain
\[
\abs{(A-g) \nomizu (A+g)} \le c \abs{\Ric - (n-1) g} \mbox{ at every point }
\]
and we can perform the following estimate.
\begin{align*}
\norm{(A-g) \nomizu (A+g)}_p
&\le c \norm{\Ric - (n-1) g}_p \le c\coup*{\norm{\tRic}_p + \norm{R - n(n-1)}_p} \\
&\le C \coup*{ \norm{\tRic}_p + \norm{R - \Rm}_p + \abs{\Rm - R_\sigma} } \\
&\le C \coup*{ \norm{\tRic}_p  + \abs{\Rm - R_\sigma} }
\end{align*}
Plugging this inequality into \eqref{WeylIneq} we obtain the desired estimate. 
\end{proof}
\subsection{Proof of proposition \ref{ApproxThm}}
Here we prove the approximation theorem \ref{ApproxThm}. The idea is to argue by contradiction, and then apply the Cheeger-Gromov compactness theorem to a proper sequence of radially parametrized hypersurfaces.
\begin{proof}
The thesis follows from the following claim: for every $\epsilon$ and $\alpha$ there exists $\delta$ with the following property. If $(\Sigma, \, g)$ satisfies the assumptions of theorem \ref{ApproxThm} and the inequality
\[
\norm{\tRic}_{p} + \norm{W}_{p}  \le \delta 
\]
then the pull-back metric $g$ on the sphere satisfies the inequality
\begin{equation}\label{ApproxMetric}
\norm{g - \sigma}_{C^{1, \, \alpha}} \le \epsilon
\end{equation}
We assume \eqref{ApproxMetric} and prove proposition \ref{ApproxThm}. Indeed, let us recall expression \eqref{g} for the metric in term of the radius $\rho$:
\[
g_{ij} = e^{2 f} \coup*{\sigma_{ij} + \nabla_i f \nabla_j f} = \rho^2 \sigma_{ij} + \nabla_i \rho \nabla_j \rho
\]
Therefore we obtain a nearness expression for $\rho$: 
\begin{equation}\label{ApproxRho}
\norm{ (\rho^2 - 1) \sigma + \nabla \rho \otimes \nabla \rho }_{C^{1, \, \alpha} } \le \epsilon
\end{equation}
Let $x_M$ and $x_m$ be two points such that $\rho(x_M) = \max \rho$, $\rho(x_m) = \min \rho >0$. We know from \ref{ApproxRho} that 
\[
\abs{(\rho^2 - 1) \sigma + \nabla \rho \otimes \nabla \rho} \le \epsilon
\]
Evaluating at $x_M$ and $x_m$, we obtain
\[
\coup{ \rho(x_M) - 1}^2, \, \coup{\rho(x_m) - 1}^2 \le \epsilon
\]
because $\nabla \rho = 0$ at $x_m$ and $x_M$. Since $\rho$ is positive, we consider the simple inequality $1 + \rho > 1$ and find a $L^\infty$-inequality for the radius.
\[
\abs{\rho(x_M) - 1}, \, \abs{\rho(x_m) - 1} \le \epsilon \Rightarrow \norm{\rho - 1 } \le \epsilon
\]
From this we easily infer a $L^\infty$-inequality for the differential of the radius.
\[
\norm{\nabla \rho \otimes \nabla \rho}_\infty = \norm{\nabla \rho}^2_\infty \le C \epsilon
\]
We transpose this inequality for the logarithmic radius $f$ and obtain inequalities \eqref{ApproxFirst} and \eqref{ApproxSecond}. Inequality \eqref{ApproxThird} follows by considering the estimate for the first derivatives. Let us argue with the logarithmic radius. We have
\begin{equation}\label{DerivativeG}
\nabla g =  \nabla f \otimes g + e^{2f} \coup*{\nabla^2 f \otimes \nabla f + \nabla f \otimes \nabla^2 f}
\end{equation}
Inequality \ref{ApproxMetric} gives us
\[
\norm{\nabla g}_\infty \le \epsilon
\]
From this inequality and expression \eqref{DerivativeG} we infer
\[
\abs{\nabla^2 f \otimes \nabla f + \nabla f \otimes \nabla^2 f}^2 = 2 \abs{\nabla^2 f} \, \abs{\nabla f} + 2 \abs{\nabla f[\nabla f]}^2 \le C \epsilon
\]
and we obtain \eqref{ApproxThird}.

We prove the claim. Let us argue by contradiction. Assume there exists a sequence $\Set{\Sigma_k}_{k \in \enne}$ of radially parametrized, admissible hypersurfaces such that 
\[
 \lim_k \norm{\tRic}_p + \norm{W}_p = 0
\]
but the pull-back metrics $g^k$ on the sphere satisfy
\[
\norm{g^k- \sigma}_{C^{1, \, \alpha}} \ge \epsilon_0
\]
We recall that a surface is radially parametrized if $\Sigma=\psi(\esse^n)$
\[
\psi \daA{\esse^n}{\Sigma_k},\ \psi(x)= e^{f(x)} \, x 
\]
and it is admissible if 
\begin{align*}
\norm{A}_\infty &\le \Lambda \\
\Vol_n(\Sigma) &= \Vol_n(\esse^n) \\ 
 \diam \Sigma &\le D 
\end{align*}
Let us consider the sequence $\set{(\esse^n, \, g^k)}_{k \in \enne}$. The sequence is contained in $\Ch_{\Lambda, \, V, \, D}$. By the compactness theorem \ref{CheegerCptness}, we can assume that up to a subsequence $g^k$ converge to a metric $g$ in $C^{1, \, \alpha}$. We show that necessarily $g=\sigma$, and this is a contradiction.
Firstly, we show that the limit metric $g$ must be Ricci flat. The proof is strongly similar to the one made in \cite[p. $192$-$193$]{PetersenCollection}, so we just give a sketch of it and leave all the nasty details to it. We fix a point $x \in \esse^n$ and consider $y^k$ harmonic coordinates for $g^k$ near $x$: as shown before, we can easily find a radius $R_0$ such that the coordinates $y^k$ are well defined in $B_{R_0}(x)$ for every $k$. We recall the expression \eqref{RicciHarm} of the Ricci tensor:
 \[
 -\frac{1}{2} \Delta_{g^k} g^k_{ij} + Q_{ij}(g^k, \, D g^k) = \Ric^{g^k}_{ij}
 \]
Passing to the limit\footnote{Notice that we are subtly implying that also $y^k \to y$ harmonic system for $g$. This claim is true and proved in \cite{PetersenCollection}.}, we obtain the following, limit equation
\[
 -\frac{1}{2} \Delta_{g} g_{ij} + Q_{ij}(g, \, D g) = \Ric^{g}_{ij} 
\]
which holds in the sense of distribution. We show that $\Ric^g=cg$. Firstly, we track the Gauss equation \eqref{GaussEq}  twice and obtain the expression for the scalar curvature of $\Sigma$:
\[
R = H^2 - \abs{A}^2
\]
Since every $\Sigma_k$ is admissible and convex, by proposition \ref{Convexity} we obtain the the scalar curvatures $\set{R^{g^k}}_{k \in \enne}$ satisfy 
\[
0 \le R^{g^{k}} \le n^2 \Lambda^2
\]
Therefore, up to extract another subsequence, we can assume the existence of a non-negative $c$ such that $\overline{R^{g^k}}  \longrightarrow c$. Plugging these result into the distributional equation, we find that the limit metric satisfies 
\begin{equation}\label{LimitEq}
 -\frac{1}{2} \Delta_{g} g_{ij} + Q_{ij}(g, \, D g) = c g_{ij} \mbox{ for every } i, \, j
\end{equation}
Therefore $g$ is analytic and Einstein, with $\Ric^g=cg$. Moreover, we can easily see that the Weyl tensor $W_{g^k}$ weakly converge to $W$, namely
\[
\int_M ( W_{g^k}, \, T) \, dV_{g^k} \longrightarrow \int_M ( W_{g}, \, T) \, dV_g \mbox{ for any $4$-covariant tensor } T
\]
This result shows us that the limit metric $g$ has also null Weyl tensor. Indeed
\[
0 \le \norm{W_g}_{L^p_g(\esse^n)} \le \liminf_k \norm{W_{g^k}}_{L^p_{g^k}(\esse^n)} = 0
\]
Equation \eqref{RiemDecompEq} ensures that an Einstein metric with null Weyl tensor has constant sectional curvature. Since the sphere cannot have a globally flat metric, we already know that the Riemann tensor is positive definite, namely
\[
\Riem^g = \frac{c}{2n(n-1)} g \nomizu g, \mbox{ with } c>0
\]
Therefore $g=\mu \sigma$ for some $\mu > 0$. We conclude the proof by showing that $\mu =1$. This is ensured by the volume condition. Indeed, we have
\[
\Vol_n(\esse^n) = \Vol_{g^k}(\esse^n) \longrightarrow \Vol_g(\esse^n)
\]
However, we can also find
\[
g=\mu \sigma \Rightarrow \sqrt{\det g} =  \mu \sqrt{\sigma} \Rightarrow \Vol_n(\Sigma) = \mu \Vol_n(\esse^n)
\]
Necessarily we discover $\mu=1$, hence the contradiction.
\end{proof}
\subsection{Proof of proposition \ref{AlmostThm}}
This is the last step. Again, we follow a strategy outlined in \cite{Gioffre2016}: we write the linearised main quantities, and then obtain approximated formulas. Using again a particular reduction to eigenvalues, we reduce to a well known case and solve it.
\begin{proof}
We recall that we are considering a convex, closed hypersurface $\Sigma$ which admits a radial parametrization 
\[
\psi \daA{\esse^n}{\Sigma}\ \psi(x)=e^{f(x)} \, x
\]
whose logarithmic radius $f$ satisfies estimate \eqref{ApproxFirst}, \eqref{ApproxSecond}, \eqref{ApproxThird}.
The starting point of our analysis is the following equation, obtained by patching equations \eqref{GaussEq} and \eqref{RiemDecompEq}:
\begin{equation}\label{StartingEq}
\frac{1}{2} A \nomizu A - \frac{\Rm}{2n(n-1)} g \nomizu g = \frac{R - \Rm}{2n(n-1)} g \nomizu g + \frac{1}{n-2} \tRic \nomizu g + W
\end{equation}
From equation \eqref{StartingEq} and theorem \ref{AverageScalarThm} we easily infer the estimate
\[
\norm*{A \nomizu A - \frac{\Rm}{n(n-1)} g \nomizu g }_{L^p_g(\Sigma)} \le C \coup*{ \norm{\tRic}_{L^p_g(\Sigma)} + \norm{W}_{L^p_g(\Sigma)} }
\]
where as usual $C=C(n, \, p, \, \Lambda, \, D)$. 
We use theorem \ref{ApproxThm} in order to simplify the left hand side.
In particular, we prove the following lemma 
\begin{lemma}\label{SimpLemma}
Under the hypothesis of theorem \ref{AlmostThm}, the average of the scalar curvature can be approximated as follows: 
\begin{equation}\label{SimpleMean}
\Rm = n(n-1) + O_{\sqrt{\epsilon}}(\norm{f}_{W^{2, \, p}_\sigma(\esse^n)})
\end{equation}
\end{lemma}
We show how the thesis follows by lemma \ref{SimpLemma}, and then prove it.  Firstly, we use equation \eqref{SimpleMean} to improve proposition \ref{Stability} and obtain
\begin{equation}\label{StabilityImproved}
\norm{W}_{L^p_g(\Sigma)} \le C \coup*{ \norm{\tRic}_{L^p_g(\Sigma)} + \sqrt{\epsilon} \norm{f}_{W^{2, \, p}_\sigma(\esse^n)} }
\end{equation}
From \eqref{SimpleMean} and \eqref{StabilityImproved} we obtain the following estimate:
\begin{equation}\label{AlmostDone}
\norm*{(A - g) \nomizu (A + g) }_{L^p_g(\Sigma)} \le C \coup*{ \norm{\tRic}_{L^p_g(\Sigma)} + \sqrt{\epsilon} \norm{f}_{W^{2, \, p}_\sigma(\esse^n)}  }
\end{equation}
Again, we study $(A - g) \nomizu (A + g) $  pointwise. Given $p \in M$ we consider coordinates such that $g= \delta$ and $A=D(x)$ at $p$. We consider the polynomial problem associated to it and prove the following lemma.
\begin{lemma}\label{LemmaPolyTre}
Let $p=p(x)$ be defined as in \eqref{PolyNomizu}. We set $r=r(x)$ as 
\begin{equation}\label{PolyResto}
r(x):= \abs{D(x) - \delta}^2  \, \abs{D(x) + \delta}^2
\end{equation}
There exist $c_2=c_2(n)$ and $c_3=c_3(n, \, \Lambda)$ such that 
\[
c_2 \le \frac{p(x)}{r(x)} \le c_3
\]
\end{lemma}
Lemma \ref{PolyResto} gives us the pontwise inequality
\[
\abs{A - g} \abs{A+g} \le c \abs{(A - g) \nomizu (A+g)}
\]
This allows us to improve \eqref{AlmostDone} as follows:
\[
\norm*{ \abs{A - g} \, \abs{A + g} }_{L^p_g(\Sigma)} \le C \coup*{ \norm{\tRic}_{L^p_g(\Sigma)} + \sqrt{\epsilon} \norm{f}_{W^{2, \, p}_\sigma(\esse^n)}  }
\]
Since $\Sigma$ is convex, by proposition \ref{Convexity} we have the punctual inequality ${A + g \ge g}$. Finally we find
\begin{equation}\label{Done}
\norm*{A - g }_{L^p_g(\Sigma)} \le C \coup*{ \norm{\tRic}_{L^p_g(\Sigma)} + \sqrt{\epsilon} \norm{f}_{W^{2, \, p}_\sigma(\esse^n)}  }
\end{equation}
We show how inequality \eqref{Done} and theorem \ref{ApproxThm} give us the thesis. We start by approximating the second fundamental form. We recall formula \eqref{Aff} for the second fundamental form $A$.  
\[
A_{ij} = \frac{e^f}{\sqrt{1 + \abs{\nabla f}^2}} \coup*{ \sigma_{ij} + \nabla_i f \, \nabla_j f - \nabla^2_{ij} f}
\]
We simplify the expression with the $W^{1, \, \infty}$-smallness of $f$ and obtain
\begin{align*}
A_{ij} 
&= e^f  \coup*{ \sigma_{ij} + \nabla_i f \, \nabla_j f - \nabla^2_{ij} f} + O_{\sqrt{\epsilon}}(\norm{\nabla f}_{1, \, p}) \\
&= e^f  \coup*{ \sigma_{ij} - \nabla^2_{ij} f} + O_{\sqrt{\epsilon}}(\norm{\nabla f}_{1, \, p}) 
\end{align*}
where we have used nothing but the  simple identity
\begin{align*}
\frac{1}{\sqrt{1 + \abs{\nabla f}^2}} -1 
&= \int^1_0 \ddt \frac{1}{\sqrt{1 + t^2\abs{\nabla f}^2}} \, dt \\
&= \abs{ \nabla f}^2 \int^1_0  \frac{2t}{\sqrt{(1 + t^2\abs{\nabla f}^2)^3}} \, dt =  O_{\sqrt{\epsilon}}(\norm{\nabla f}_p)
\end{align*}
With the same idea for we approximate the exponential: 
\[
e^{f} = 1+ \int_0^1 \ddt e^{tf} \, dt = 1+ f \int_0^1 e^{tf} \, dt =1 + f + O_\epsilon (\norm{f}_p)
\]
These simplifications give us the approximated second fundamental form:
\begin{equation}\label{SimpleAff}
A_{ij} = \sigma_{ij} - \nabla^2_{ij} f + f \sigma_{ij} + O_{\sqrt{\epsilon}}(\norm{ f}_{2, \, p})
\end{equation}
With the same ideas we find also the approximated metric
\begin{equation}\label{SimpleG}
g_{ij} = (1 + 2 f+ O_{\sqrt{\epsilon}}(\norm{f}_{1, \, p}) )  \sigma_{ij} 
\end{equation}
and its inverse 
\begin{equation}\label{SimpleGInv}
g^{ij} = (1 - 2 f+ O_{\sqrt{\epsilon}}(\norm{f}_{1, \, p}) )   \sigma^{ij} 
\end{equation}
We plug these simplified expression into \eqref{Done}, and finally obtain 
\begin{equation}
\norm*{\Delta_\sigma f + n f}_{L^p_\sigma(\Sigma)} \le C \coup*{ \norm{\tRic}_{L^p_g(\Sigma)} + \sqrt{\epsilon} \norm{f}_{W^{2, \, p}_\sigma(\esse^n)}  }
\end{equation}
And the thesis follows by lemma \ref{ObataLike}.
\end{proof}
We prove lemma \ref{SimpLemma} and complete this section.
\begin{proof}[Proof of lemma \ref{SimpLemma}]
Firstly, we have to find an approximation of the scalar curvature $R$. As done before, we track equation \eqref{GaussEq} twice and obtain 
\[
R = H^2 - \abs{A}^2
\]
Therefore, we have to find an approximate expression for $H^2$ and $\abs{A}^2$. With little effort, we find the approximated expression of $H^2$ and $\abs{A}^2$:
\begin{align}
H^2 
&= n^2 - 2n \Delta_\sigma f + \coup*{\Delta f}^2 - 2n^2 f + O_{\sqrt{\epsilon}}(\norm{f}_{2, \, p}) \\
\abs{A}^2 
&= n - 2 \Delta_\sigma f + \abs{\nabla^2 f}^2 - 2n f + O_{\sqrt{\epsilon}}(\norm{f}_{2, \, p})
\end{align}
From these two expressions we find the approximated curvature:
\begin{align}\label{SimpleCurv}
R &= n(n-1) - 2(n-1) \Delta_\sigma f + \coup*{\Delta f}^2 - \abs{\nabla^2 f}^2 + \\
 &- 2n(n-1) f + O_{\sqrt{\epsilon}}(\norm{f}_{2, \, p}) \notag
\end{align}
Now we integrate and study $\Rm$. We recall that although the integral is computed with respect to the measure $dV_g$, we can easily pass from $dV_g$ to $dV_\sigma$ by approximating the density $\faktor{dV_g}{dV_\sigma}$. Indeed, we can easily write
\[
e^{nf} \sqrt{1 + \abs{\nabla f}^2} = 1 + n f + O_{\sqrt{\epsilon}}(\norm{f}_{1, \, p})
\]
From this expression we obtain
\begin{align*}
 \Rm &= \fint R \, dV_g = n(n-1) +  \fint \coup*{\Delta f}^2 - \abs{\nabla^2 f}^2  \, dV_\sigma  + \\
&+  2n(n-1)\fint f  \, dV_\sigma + O_{\sqrt{\epsilon}}(\norm{f}_{2, \, p}) 
\end{align*}
We simplify the second-order terms. Indeed, by definition \eqref{RiemannSign} of Riemann tensor we know the commutation formula
\[
\nabla_j \nabla_i \alpha_k - \nabla_i \nabla_j \alpha_k = \Riem^l_{ijk} \alpha_l
\]
Using this formula, integration by parts shows us how in any closed manifold the following formula holds.
\begin{align*}
\int (\Delta f)^2 \, dV 
&= \int \nabla^i \nabla_i f \cdot \nabla^j \nabla_j f \, dV = - \int \nabla_i f \cdot \nabla^i \nabla^j \nabla_j f \, dV  \\
&=  - \int \nabla_i f \cdot \nabla^j \nabla^i \nabla_j f \, dV + \int \Ric(\nabla f, \, \nabla f) dV \\
&= \int  \nabla^j  \nabla_i f \cdot \nabla^i \nabla_j f \, dV + \int \Ric(\nabla f, \, \nabla f) dV \\ 
&= \int  \abs{\nabla^2 f}^2 \, dV + \int \Ric(\nabla f, \, \nabla f) dV 
\end{align*}
In the case of the sphere, this computation gives us the equality
\[
\int (\Delta_\sigma f)^2 \, dV_\sigma  - \int  \abs{\nabla^2 f}^2 \, dV_\sigma = (n-1) \int \abs{\nabla f}^2 dV_\sigma = O_{\sqrt{\epsilon}}(\norm{\nabla f}_p) 
\]
And we can improve the expression for the mean curvature, obtaining
\[
 \Rm = n(n-1) +  2n(n-1)\fint f  \, dV_\sigma + O_{\sqrt{\epsilon}}(\norm{f}_{2, \, p}) 
\]
We just have to show that the average of the logarithmic radius is negligible. The goal is achieved by using the volume condition. By volume formula \eqref{Volume} we know
\[
1 = \frac{\Vol_n(\Sigma)}{\Vol_n(\esse^n)} = \fint_{\esse^n} e^{nf} \sqrt{1 + \abs{\nabla f}^2} \, dV
\]
However we can perform the following approximation:
\[
\fint e^{nf} \sqrt{1 + \abs{\nabla f}^2} \, dV_\sigma = 1 + n \fint f \, dV_\sigma + O_{\sqrt{\epsilon}}(\norm{f}_{1, \, p}) 
\]
Plugging this approximation into the volume equality, we obtain 
\[
\fint f \, dV_\sigma = O_{\sqrt{\epsilon}}(\norm{f}_{1, \, p}) 
\]
and this proves the lemma.
\end{proof}
\section{Proof of the computational lemmas}
We conclude the paper by proving the computational lemmas.  
Let us recall the polynomials we are going to study. 
\begin{align*}
p(x) &:= \abs{(D(x)-\delta) \nomizu (D(x)+\delta)}^2  \\
q(x) &:= \abs{\Ric(x) - (n-1)\delta}^2 \\
r(x) &:= \abs{(D(x)-\delta)}^2 \abs{(D(x)+\delta)}^2 
\end{align*}
We make the lemma follow by a fine study of the quotients $\faktor{p(x)}{r(x)}$ and $\faktor{q(x)}{r(x)}$.
Firstly, We show that the only zeros of $p$ and $q$ are $\coup{1, \, \dots \, 1}$ and $\coup{-1, \, \dots \, -1}$. 
\subsubsection*{Zeros of $p$}
With little effort, we obtain the following expression for $p$:
\begin{equation}\label{PolyNomizuDue}
p(x)= \coup*{ \abs{x}^4 + \sum_{i=1}^n x_i^4 } - 2 \coup*{H - \abs{x}^2} + n(n-1)
\end{equation}
where we have set 
\begin{align*}
\abs{x^2}&:= \sum_{i=1}^n x_i^2 \\
H &:= \sum_{i=1}^n x_i
\end{align*}
We want to compute the zeros of $p$. Since $p$ is positive, these zeros must also be minima for $p$, and therefore must satisfy $D p = 0$. The computation of $D_i p$ is quite simple: 
\[
D_i p = 4 \coup*{ \abs{x}^2 x_i - x_i^3 + x_i - H }
\]
We impose $D_i p = 0$ and obtain the equation 
\begin{equation}\label{ZerosNomizu}
\abs{x}^2 x_i - x_i^3 + x_i = H 
\end{equation}
We easily notice that $z=0$ satisfies $Dp(z)=0$. However, $p(0)= n(n-1)$ tells us that $0$ is not a minimum. Let us assume $z \neq 0$, and let us write $z= \coup*{z_i, \, \dots z_n}$. We firstly show that all the coordinates $z_i$ are non-zero. Let us assume by contradiction that there exists  $k \in \set{1, \, \dots \, n}$. If we consider equation \eqref{ZerosNomizu} for $i=k$, we obtain 
\[
H=0
\]
However, since $z \neq 0$ we can also find a $j \in \set{1, \, \dots \, n}$ such that $z_j \neq 0$. For that $j$ we obtain 
\begin{equation}\label{Imp}
\abs{z}^2 - z_j^2 + 1 = 0
\end{equation}
but this equation has clearly no solutions: therefore $z_j \neq 0$ for every $j$.

The next step is to show that $\sgn z_i = \sgn z_j$ for every $i$ and $j$. Indeed, we obtain for every $i$
\[
z_i\coup*{ \abs{z}^2 - z_i^2 + 1} = H
\]
As shown in \eqref{Imp}, the quantity between parenthesis is positive, therefore $\sgn z_i = \sgn H$ and this proves the second step. 

We claim that a critical point must satisfy $z_i = z_j$ for every $i, \, j$. If the claim is true, by \eqref{ZerosNomizu} we obtain that a critical point $z=t(1, \, \dots \, 1)$ must satisfy
\[
(n-1)t^3 = (n-1)t
\]
and this concludes the proof. Let us prove the claim. Firstly we write equation \eqref{ZerosNomizu} as follows.
\begin{equation}\label{ZerosNomizuDue}
\abs{z}^2 + 1 = \frac{1}{z_i} H + z_i^2
\end{equation}
Let us assume $z_i \neq z_j$. From \eqref{ZerosNomizuDue} we infer
\[
\frac{1}{z_i} H + z_i^2 = \frac{1}{z_j} H + z_j^2
\]
and from this equation we obtain 
\begin{equation}\label{ZeroNomizuTre}
z_i \, z_j (z_i + z_j) = H
\end{equation}
Assume $z_k \neq z_i$ for $k \neq i, \, j$. We can make the same computation made before with indexes $i$, $k$ and obtain equation \eqref{ZeroNomizuTre}
\begin{equation}\label{ZeroNomizuQuattro}
z_i \, z_k(z_i + z_k) = H
\end{equation}
We equalize \eqref{ZeroNomizuTre}, \eqref{ZeroNomizuQuattro} and obtain
\[
z_i (z_k - z_j) = (z_j + z_k) (z_j - z_k)
\]
which implies $z_j = z_k$ because the signs of $z_j - z_k$ and $z_k - z_j$ are opposite. Hence a critical point can have at most two different values $a$ and $b$. Assume $a \neq b$, and suppose $a$ appear $k$ times and $b$ appear $n-k$ times in the coordinates of $z$. We rewrite equations \eqref{ZerosNomizu}, \eqref{ZeroNomizuTre} in term of $a$ and $b$, and obtain:
\begin{align*}
H &= a ( (k-1)a^2 + (n-k)b^2 + 1)  \\
H &= b ( k a^2 + (n-k-1)b^2 + 1)  \\
H &= ab(a+b) 
\end{align*}
We can use the third equation to simplify the other two, and obtain the system
\begin{align*}
b(a + b) &= (k-1)a^2 + (n-k)b^2 + 1 \\
a(a + b) &=  k a^2 + (n-k-1)b^2 + 1 
\end{align*}
Summing these two equations, we obtain:
\begin{equation}\label{ZeroNomizuLast}
(k-1)a^2 + (n-k-1)b^2 + 1 = ab
\end{equation}
We are done: if both $k-1$ and $n-k-1$ are non-zero, then we find a contradiction by applying the Cauchy-Schwarz inequality. If one of them is null, say $k=1$, then from equation \eqref{ZerosNomizu} we obtain 
\[
(n-1)ab^2 + a = a + (n-1)b \Rightarrow ab=1
\]
and plugging the last equality into \eqref{ZeroNomizuTre} we obtain 
\[
(n-2)b = 0
\]
but we have assumed $a$, $b$ to be $\neq 0$. Hence we obtain the thesis.
\subsubsection*{Zeros of $q$}
We show here that the only zeros of $q$ are $(1, \, \dots \, 1)$ and $-(1, \, \dots \, 1)$, and therefore obtain proposition \ref{Stability}. More precisely, we study the system 
\begin{equation}\label{PolyRicciZero}
\Ric(x) = (n-1)\delta 
\end{equation}
and prove that $(1, \, \dots \, 1)$ and $-(1, \, \dots \, 1)$ are the only solution for this system. We recall the expression of $\Ric(x)$:
\[
\Ric(x)= \tr D(x) \, D(x) - D(x)^2 
\]
Hence, system \eqref{PolyRicciZero} can be written as
\begin{equation}\label{PolyRicci}
H x_i - x_i^2 = (n-1)
\end{equation}
where we have set 
\[
H := \sum_{i=1}^n x_i
\]
Consider a point $z$ which is a solution for \eqref{PolyRicci}. Again, we claim that $z_i=z_j$ for every $i$, $j$. If the claim is true, system \eqref{PolyRicci} for $z=t(1, \, \dots \, 1)$ is reduced to 
\[
(n-1)t^2 = (n-1)
\]
and this proves proposition \ref{Stability}. Let us assume by contradiction that there exist two indexes $i$, $j$ such that $z_i \neq z_j$. We can equalise the columns of system \ref{PolyRicci} and obtain
\begin{equation}\label{PolySum}
H z_i - z_i^2 = H z_j - z_j^2 \Rightarrow H = z_i + z_j
\end{equation}
Substituting the last expression in the system, we obtain 
\begin{align*}
n-1 &= (z_i + z_j) z_i - z_i^2 \\ 
n-1 &= (z_i + z_j) z_j - z_j^2
\end{align*}
from whom we easily infer
\begin{equation}\label{PolyProd}
z_i z_j = (n-1)
\end{equation}
Consider $z_k \neq z_i$. We can repeat the computations made before, and obtain the equality
\[
z_i z_k = (n-1)
\]
and this clearly implies $z_k=z_j$. Therefore, any critical point $z$ must have at most two different values. Call them $a$ and $b$, and assume $a$ appears $k$ times and $b$ appears $n-k$ times in the coordinate of $z$. From equality \eqref{PolySum} we have 
\[
(k-1) a + (n-k-1) b = 0
\]
If both $k-1$ and $n-k-1$ are $\neq 0$, then $a$ and $b$ must have different sign, and equation \eqref{PolyProd} is violated. If one of them is $0$, say $k-1=0$, then we must have $b=0$, but again equation \eqref{PolyProd} would be violated again. Hence all the values are equal, and we easily find the thesis of proposition \ref{Stability}.
\subsubsection*{Proofs of lemmas \ref{LemmaPolyDue} and \ref{LemmaPolyTre}}
We make the lemmas follow by a fine study of $\faktor{p}{r}$ and $\faktor{q}{r}$. Both the quotients are continuous in the whole space except that in $(1, \, \dots \, 1)$ and in $-(1, \, \dots \, 1)$. We claim that the limit superior and the limit inferior of $\faktor{p}{r}$ and $\faktor{q}{r}$ are bounded near these points, and hence we conclude the study. Indeed, if we are able to find constants $c_0=c_0(n)$ and $c_2=c_2(n)$ such that 
\[
c_0 \le \frac{p}{r}, \quad  c_2 \le \frac{q}{r}
\]
then we are done, since the assumption $\norm{A}_\infty \le \Lambda$ implies that the study of the quotients has to be made inside the ball $B_{\Lambda}$. Therefore, we are able to find constants $c_1=c_1(n, \, \Lambda)$, $c_3=c_3(n, \, \Lambda)$ such that
\[
 \frac{p}{r}\le c_1, \quad    \frac{q}{r} \le c_3
\]
and both the lemmas follow by comparison with $r$. We now prove the claim. Let us start with $\faktor{p}{r}$. Firstly, we notice that due to clear symmetries we can study only the limit in $\coup{1, \, \dots \, 1}$. Let us write $y_i = x_i -1$. We obtain the quotient in the following form
\begin{align*}
\frac{p(y)}{r(y)} &= \frac{\abs{y}^2\coup*{\sum_i (y_i + 2)^2} + \sum_i \coup*{ y_i (y_i + 2)}^2 - 2 \sum_i y_i^2(y_i + 2)^2  }{\abs{y}^2 \coup*{\sum_i (y_i + 2)^2}}
\end{align*}
This shows that near $y=0$ we can make the approximation 
\[
\frac{p(y)}{r(y)} =\frac{(n-2) \abs{y}^2 + \coup*{ \sum_i y_i}^2  }{n \abs{y}^2} + O(\abs{y})
\]
And this proves the inequalities:
\[
\frac{n-2}{n} + O(\abs{y)} \le \frac{p(y)}{r(y)} \le \frac{n-2}{n} + \frac{1}{n} + O(\abs{y}) \mbox{ near } 0
\]
We have obtained a constant $c_0=c_0(n)$ such that 
\[
c_0 \le \frac{p}{r}
\]
and this complete the first study. We complete the paper by dealing with $\faktor{q}{r}$. In order to obtain a lower bound, we write $q$ in a clever way. 
\begin{align*}
 \Ric_{ij} - (n-1)\delta_{ij} 
&= (H - n)(D_{ij - \delta_{ij}}) + n (D_{ij} - \delta_{ij}) + (H - n )\delta_{ij} + \\
&- \delta^{pq} (D_{ip - \delta_{ip}})(D_{jq} - \delta_{jq}) 
\end{align*}
Again, the symmetries of the problem allow us to consider only one limit. Hence, we write $y_i = x_i - 1$ and obtain
\[
q(y) = (n-2)^2 \abs{y}^2 + (3n - 4)H^2 + O(\abs{y}^3)
\]
This gives us a constant $c_2$ such that 
\[
c_2 \le \frac{q}{r}
\]
Hence $p$, $q$ and $r$ are comparable, and lemmas \ref{LemmaPolyDue} and \ref{LemmaPolyTre} are proved.
\section*{Acknowledgements}
The author wishes to thanks Camillo de Lellis who brought the problem to his attention and Elia Brué for its help in some of nastiest parts of the computations. 
\bibliographystyle{plain}
\bibliography{Problem}

\begin{thebibliography}{10}

\bibitem{Ambrosio}
Luigi Ambrosio.
\newblock Lecture notes on partial differential equations.
\newblock \url{http://cvgmt.sns.it/paper/1280/}, 2010.

\bibitem{Anderson}
Michael~T. Anderson.
\newblock Convergence and rigidity of manifolds under {R}icci curvature bounds.
\newblock {\em Invent. Math.}, 102(2):429--445, 1990.

\bibitem{Cheeger}
Jeff Cheeger.
\newblock Finiteness theorems for {R}iemannian manifolds.
\newblock {\em Amer. J. Math.}, 92:61--74, 1970.

\bibitem{CheegerEbin}
Jeff Cheeger and David~G. Ebin.
\newblock {\em Comparison theorems in {R}iemannian geometry}.
\newblock North-Holland Publishing Co., Amsterdam-Oxford; American Elsevier
  Publishing Co., Inc., New York, 1975.
\newblock North-Holland Mathematical Library, Vol. 9.

\bibitem{Fialkow}
Aaron Fialkow.
\newblock Hypersurfaces of a space of constant curvature.
\newblock {\em Ann. of Math. (2)}, 39(4):762--785, 1938.

\bibitem{GHL}
Sylvestre Gallot, Dominique Hulin, and Jacques Lafontaine.
\newblock {\em Riemannian geometry}.
\newblock Universitext. Springer-Verlag, Berlin, third edition, 2004.

\bibitem{Gioffre2016}
Stefano {Gioffr{\`e}}.
\newblock {A $W^{2, \, p}$-estimate for nearly umbilical hypersurfaces}.
\newblock {\em ArXiv e-prints}, December 2016.

\bibitem{Jost}
J\"urgen Jost.
\newblock {\em Riemannian geometry and geometric analysis}.
\newblock Universitext. Springer-Verlag, Berlin, fourth edition, 2005.

\bibitem{DLT}
Camillo~De Lellis and Peter~M. Topping.
\newblock Almost-{S}chur lemma.
\newblock {\em Calc. Var. Partial Differential Equations}, 43(3-4):347--354,
  2012.

\bibitem{Daniel}
Daniel Perez.
\newblock On nearly umbilical surfaces, 2011.

\bibitem{PetersenCollection}
Peter Petersen.
\newblock Convergence theorems in {R}iemannian geometry.
\newblock In {\em Comparison geometry ({B}erkeley, {CA}, 1993--94)}, volume~30
  of {\em Math. Sci. Res. Inst. Publ.}, pages 167--202. Cambridge Univ. Press,
  Cambridge, 1997.

\bibitem{Petersen}
Peter Petersen.
\newblock {\em Riemannian geometry}, volume 171 of {\em Graduate Texts in
  Mathematics}.
\newblock Springer, New York, second edition, 2006.

\bibitem{Roth}
Julien Roth.
\newblock Pinching of the first eigenvalue of the {L}aplacian and
  almost-{E}instein hypersurfaces of the {E}uclidean space.
\newblock {\em Ann. Global Anal. Geom.}, 33(3):293--306, 2008.

\bibitem{Spivak}
Michael Spivak.
\newblock {\em A comprehensive introduction to differential geometry. {V}ol.
  {IV}}.
\newblock Publish or Perish, Inc., Wilmington, Del., second edition, 1979.

\bibitem{Thomas}
T.~Y. Thomas.
\newblock On {C}losed {S}paces of {C}onstant {M}ean {C}urvature.
\newblock {\em Amer. J. Math.}, 58(4):702--704, 1936.

\end{thebibliography}

\end{document}